\documentclass[12pt,reno]{amsart}
\usepackage{amssymb}
\usepackage{amsmath}
\usepackage{enumitem}
\usepackage{mathrsfs}
\usepackage[english]{babel}
\usepackage[all]{xy}
\usepackage{hyperref}
\usepackage{marginnote}

\usepackage{stmaryrd}

\usepackage{fullpage}

\RequirePackage{mathrsfs} %\let\mathcal\mathcal

\newtheorem{theorem}{Theorem}[section]
\newtheorem{lemma}[theorem]{Lemma}
\newtheorem{proposition}[theorem]{Proposition}

\theoremstyle{definition}
\newtheorem*{ack}{Acknowledgements}

\newtheorem{remark}[theorem]{Remark}

\numberwithin{equation}{section} \numberwithin{figure}{section}

\DeclareMathOperator{\Spec}{Spec}

\DeclareMathOperator{\an}{an}

\DeclareMathOperator{\Hom}{Hom}

\newcommand\CC{\mathbb{C}}

\usepackage{color}

\definecolor{orange}{rgb}{1,0.5,0}

\title[Finiteness of non-constant maps  over a number field]{Finiteness of non-constant maps   over a number field}

\author{Ariyan Javanpeykar}
\address{Ariyan Javanpeykar \\
Institut f\"{u}r Mathematik\\
Johannes Gutenberg-Universit\"{a}t Mainz\\
Staudingerweg 9, 55099 Mainz\\
Germany.}
\email{peykar@uni-mainz.de}

 %hyperbolicity


\begin{document}

\maketitle
% \tableofcontents

\thispagestyle{empty}
 
 \begin{abstract}  Motivated by the intermediate Lang conjectures on hyperbolicity and rational points, we prove new finiteness results for non-constant morphisms from a fixed variety to a fixed variety defined over a number field by applying Faltings's finiteness results to moduli spaces of maps.
\end{abstract}

 \section{Introduction}  
We prove   finiteness results  motivated by Lang's intermediate conjectures \cite[\S2]{EJR}; these conjectures generalize Lang's conjectures \cite{Lang2} and relate the notion of being of general type to the behaviour of  rational points. The main   idea of this note is to apply Faltings's arithmetic finiteness results (\emph{formerly} Mordell's conjecture and the Mordell-Lang conjecture) to moduli spaces of maps (i.e., Hom-schemes).

We work with the following conventions. If $K$ is a field, a variety over $K$ is a geometrically integral quasi-projective scheme over $K$. Moreover, a surface over $K$ is a two-dimensional variety and a curve over $K$ is a one-dimensional variety over $K$. If $V$ is a variety over $K$ and $L/K$ is a field extension, we let $V_L$ denote $V\times_{\Spec K} \Spec L$. Also, if $f:V\dashrightarrow X$ is a rational map, we let $f(V)$ denote the image of the locus of determinacy of $f$.
 
 Recall that a variety $X$ over $\mathbb{C}$ is \emph{Brody hyperbolic} if every holomorphic map $\mathbb{C}\to X^{\an}$ is constant, where $X^{\an}$ denotes the complex-analytic space associated to $X$.
 Our first result says, in particular, that a Brody hyperbolic projective surface admits only finitely many  non-constant maps defined over a fixed number field from a given variety.  
 
 \begin{theorem}\label{thm:brody}
Let $X$ be a   projective     surface   over a finitely generated field $K$ of characteristic zero such that there is an embedding $K\to \mathbb{C}$ with $X_{\mathbb{C}}$  Brody hyperbolic. Then, for every finitely generated extension $L/K$ and every variety $V$ over $L$, the set of non-constant rational maps $f:V\dashrightarrow X_L$ over $L$  is finite.
\end{theorem}

Note that the analogous statement over algebraically closed fields fails, so that the assumption that $K$ is finitely generated can not be dropped in Theorem \ref{thm:brody}. For example, if $C$ is a smooth projective curve of genus at least two over $\mathbb{C}$ and $X=C\times C$, then the set of non-constant morphisms $C\to X$  is (obviously) infinite.

Since a smooth projective variety over $\mathbb{C}$ with ample cotangent bundle is Brody hyperbolic, we directly obtain the following finiteness result.

\begin{theorem}\label{thm:omega1}
Let $X$ be a   smooth projective     surface   over a finitely generated field $K$ of characteristic zero such that    $\Omega^1_X$ is ample. Then, for every finitely generated extension $L/K$ and every variety $V$ over $L$, the set of non-constant rational maps $f:V\dashrightarrow  X_L$ over $L$  is finite.
\end{theorem}

Our next result establishes a similar finiteness result for certain surfaces of general type. The surfaces we will consider are neither Brody hyperbolic nor have ample cotangent bundle, so that we can not appeal to the above results. On the other hand, the surfaces we consider are "pseudo-algebraically hyperbolic" \cite[Definition~9.2]{JBook}, and it turns out that this suffices to prove a similar finiteness result.

 \begin{theorem}\label{thm:bogomolov}
Let $X$ be a smooth projective surface of general type with $c_1^2(T_X) > c_2(T_X)$ over a finitely generated field $K$ of characteristic zero. Then, there is a  proper closed subset $Z\subset X$ such that, for every finitely generated extension $L/K$ and every variety $V$ over $L$, the set of non-constant rational maps $f:V\dashrightarrow X_L$ over $L$ with $f(V)\not\subset Z_L$ is finite.  
\end{theorem} 

One can not take $Z$ to be the empty set in Theorem \ref{thm:bogomolov}. Indeed, there exists a smooth projective surface of general type with $c_1^2(T_X)> c_2(T_X)$ which contains rational curves. The subset $Z$ in Theorem \ref{thm:bogomolov} will forcefully contain these rational curves. In fact, the subset $Z$ can be taken to be the union of images of non-constant maps $E\to X_{\overline{K}}$, where $E$ runs over all elliptic curves over $\overline{K}$.

To prove the above three results we rely on two  abstract  results (Theorems \ref{thm:abstract2} and \ref{thm:abstract} below). These results may be considered as the main technical results of this note.

Before we state our results, we introduce   notation for Hom-schemes and also take the opportunity to briefly recall some fundamental properties of such schemes.  Namely, for $X$ and $Y$ projective schemes over a field $k$,  Grothendieck proved that the set-valued functor 
\[
\mathrm{Sch}/k^{\textrm{op}} \to \mathrm{Sets}
\] which associates to a $k$-scheme $T$ the set $\mathrm{Isom}_T(X_T, Y_T)$ of isomorphisms $X_T\to Y_T $ over $T$ is representable by a locally finite type $k$-scheme   $\underline{\Hom}_k(X,Y)$; Grothendieck's proof uses the existence   of Hilbert schemes \cite{FGA} (see  also   \cite[Theorem~6.6]{Nitsure} for a proof).   We  will  refer to $\underline{\Hom}_k(X,Y)$ as the \emph{moduli space} (or: \emph{Hom-scheme}) of morphisms from $X$ to $Y$; note that   the set of $k$-points $\underline{\Hom}_k(X,Y)(k)$ of $\underline{\Hom}_k(X,Y)$   is the set of morphisms from $X$ to $Y$ over $k$. Since the Hom-scheme $\underline{\Hom}_k(X,Y)$  is an open subscheme of the Hilbert scheme $\mathrm{Hilb}_{X\times_k Y/k}$ of $X\times_k Y$,  it is a countable disjoint union of quasi-projective schemes over $k$ (indexed by Hilbert polynomials). We stress that $\underline{\Hom}_k(X,Y)$ might   not be quasi-projective, as it might have infinitely many connected components (e.g., $\underline{\Hom}_k(\mathbb{P}^1_k,\mathbb{P}^1_k)$).

We define $\underline{\Hom}^c_k(X,Y)\subset \underline{\Hom}_k(X,Y)$ 
to be the subscheme  parametrizing the constant maps from $X$ to $Y$.  Note that  this subscheme of $\underline{\Hom}_k(X,Y)$ is isomorphic to $Y$. We let $\underline{\Hom}_k^{nc}(X,Y)$ be the moduli space of non-constant morphisms from $X$ to $Y$, and note that this is the complement of $\underline{\Hom}_k^c(X,Y)$. Finally, if $Z\subset Y$ is a closed subscheme, then the natural morphism $\underline{\Hom}_k(X,Z)\to\underline{\Hom}_k(X,Y)$ is a closed immersion.  The scheme $\underline{\Hom}_k(X,Y)\setminus \underline{\Hom}_k(X,Z)$ parametrizes morphisms $f:X\to Y$ such that $f(X)\not\subset Z$.

We can now state the first main technical result.

 \begin{theorem} \label{thm:abstract2} Let $X$ be a projective surface over  an algebraically closed field $k$ of characteristic zero   and let $Z\subsetneq X$ be a proper closed subset. 
 
 Assume that, for  every curve $C$ over $k$, every $c$ in $C(k)$,  and every $x$ in $X(k)\setminus Z(k)$, the set of morphisms $f:C\to X$ with $f(c)=x$ is finite. 
 
 If $C$ is a  smooth projective curve over $k$ and $H$ is a   positive-dimensional irreducible (reduced) component  of the scheme $ \underline{\Hom}^{nc}_k(C,X)\setminus \underline{\Hom}_k(C,Z)$, then $H$  is birational to a smooth projective curve of genus at least two over $k$. 
\end{theorem} 

Note that this   result is concerned with the structure of the moduli space of maps from a curve to a "pseudo-geometrically hyperbolic" projective surface \cite[Definition~11.2]{JBook}.

The proof of Theorem \ref{thm:abstract2} relies on Kobayashi-Ochiai's finiteness theorem for dominant maps from a variety to a variety of general type (Theorem \ref{thm:ko}). We apply it to the (a priori infinitely many) evaluation maps with domain the moduli space of maps from a curve $C$ to $X$ to control its dimension.

Our second technical result requires a slightly stronger, but conjecturally equivalent assumption (see (8) and (9) in \cite[Conjecture~12.1] {JBook}), and provides an arithmetic finiteness conclusion.

\begin{theorem} \label{thm:abstract}   Let $X$ be a projective surface over a finitely generated field $K$ of characteristic zero    and let $Z\subsetneq X$ be a proper closed subset.

 Assume that there is an embedding  $K\to \mathbb{C}$ such that, for  every smooth projective curve $C$  over $\CC$,  the scheme $\underline{\Hom}_{\CC}(C,X_{\CC})\setminus \underline{\Hom}_{\CC}(C, Z_{\CC})$ is of finite type over $\CC$. 
 
 Then, for every finitely generated field extension $L/K$ and every variety $V$ over $L$,   the set of non-constant rational maps $f:V\dashrightarrow X_L$ with $f(V)\not\subset Z_L$ is finite.
\end{theorem}

 Using the terminology of \cite{JBook} and \cite{EJR}, Theorem \ref{thm:abstract} implies that, 
if $X_{\CC}$ is $1$-bounded modulo $Z_{\CC}$ over $\CC$, then  $X$ is $p$-Mordellic modulo $Z$ for every $p>0$.

Our proof of Theorem \ref{thm:abstract} relies on Faltings's theorem (formerly Mordell's conjecture) and Theorem \ref{thm:abstract2}. In fact, we apply Faltings's theorem to the moduli space of non-constant maps from $V$ to $X$ in the case that $\dim V=1$. We then use a cutting argument to prove the desired finiteness statement for all varieties (by induction on the dimension of $V$). Here, it is imperative that  Theorem \ref{thm:abstract2} holds over \emph{all} finitely generated fields of characteristic zero, and does not involve any  restriction to number fields.  

We stress that  no arithmetic methods are  employed in our proofs. The main observation is simply that using known arithmetic finiteness results (due to Faltings) and algebro-geometric arguments, one can obtain finiteness results on non-constant rational maps $V\dashrightarrow X$ defined over a finitely generated field of characteristic zero.

As already explained above, all of our  results (except for Theorem \ref{thm:abstract2}) rely on Faltings's finiteness theorems for curves (i.e., Mordell's conjecture). By exploiting Faltings's finiteness theorem for closed subvarieties of abelian varieties (i.e., the Mordell-Lang conjecture), we can also obtain new finiteness results for  ramified covers of   simple abelian varieties.

   More precisely, our final  result is concerned with rational points on finite surjective \textbf{ramified} covers $\pi:X\to A$ of an abelian variety $A$ over a finitely generated field $K$ of characteristic zero with $X$ a normal variety over $K$. As explained in the introduction of \cite{CDJLZ}, Lang's conjecture implies that  the $K$-rational points $X(K)$ on such a variety $X$ is \textbf{not} dense in $X$. Moreover, if $A$ is assumed to be simple, then Lang's conjecture even predicts that $X(K)$ is finite. Neither of these conjectures are currently   known. 
 (On the positive side,  in \cite{CDJLZ} it was shown that $X$ has "fewer" points than the abelian variety it covers. More precisely: if $A(K)$ is dense, then $A(K)\setminus \pi(X(K))$ is still dense.)
 
Our final result provides a   bit of evidence for Lang's prediction for ramified covers of simple abelian varieties.  

\begin{theorem}\label{thm:ram}
Let $A$ be a geometrically simple abelian variety  over  a finitely generated field $K$ of characteristic zero, and let $X\to A$ be a finite surjective non-\'etale morphism with $X$ a normal variety over $K$. Then,   for every finite extension $L/K$ and every variety $V$ over $L$, the set of non-constant rational maps $f:V\dashrightarrow X_L$ over $L$   is finite.
\end{theorem}

  Note that, in Theorem \ref{thm:ram}, we do not assume the variety $X$ to be two-dimensional. Nonetheless, the arguments used to prove Theorem \ref{thm:ram} are similar to the arguments used to prove Theorems \ref{thm:abstract2} and \ref{thm:abstract}.

We emphasize that, in our results above, we do \textbf{not} prove finiteness of "$L$-points on $X$", but rather finiteness of "$L$-curves on $X$" (and $L$-surfaces, $L$-threefolds, etc.). This is, as noted above, quite natural from the perspective of Lang's conjectures on intermediate (pseudo-)hyperbolicity; see \cite[\S 2]{EJR}. It also fits in well with predictions made by the ``Persistence Conjecture'' (see Remark \ref{remark:pers}).

  \begin{ack}
  We thank the referee for several useful comments.
  \end{ack}

\section{Proof of Theorems \ref{thm:abstract2} and \ref{thm:abstract}}

We will use Kobayashi-Ochiai's finiteness theorem \cite[Theorem~7.6.1]{KobayashiBook}.

\begin{theorem}[Kobayashi-Ochiai]\label{thm:ko} Let $k$  be an algebraically closed field of characteristic zero.
If $X$ is a projective variety  of general type over $k$ and $Y$ is a projective variety over $k$, then the set of dominant rational maps $Y\dashrightarrow X$ is finite.
\end{theorem}

% \begin{theorem} \label{thm:abstract2} Let $X$ be a projective surface over  an algebraically closed field $k$ of characteristic zero   and let $Z\subsetneq X$ be a proper closed subset. Assume that, for  every variety $W$ over $k$, every $w$ in $W(k)$,  and every $x$ in $X(k)\setminus Z(k)$, the set of morphisms $f:W\to X$ with $f(w)=x$ is finite. If $C$ is a  smooth projective curve over $k$, then every positive-dimensional irreducible component $H\subset \underline{\Hom}^{nc}_k(C,X)\setminus \underline{\Hom}_k(C,Z)$   is birational to a smooth projective curve of genus at least two over $k$. 
%\end{theorem}

We now prove our first technical result.

\begin{proof}[Proof of Theorem \ref{thm:abstract2}] Our assumption says that $X$ is pseudo-geometrically hyperbolic  \cite[Definition~11.2]{JBook}. In particular, since $\dim X =2$, by combining \cite[Proposition~5.4]{JXie} and \cite[Lemma~3.23]{JXie}, it follows that $X$ is of general type.

As in the statement of the theorem, let $H$ be a positive-dimensional irreducible (reduced) component of the scheme $\underline{\Hom}^{nc}_k(C,X)\setminus \underline{\Hom}_k(C,Z)$. Then, since every connected component of $\underline{\Hom}_k(C,X)$ is quasi-projective, the irreducible scheme  $H$ is   quasi-projective over $k$.

For $c$ in $C$, let $\mathrm{ev}_c:H\to X$ be the evaluation-at-$c$ map. By our assumption, this morphism has finite fibres over $X\setminus Z$. Now, let $c$ be in $C$ be such that $\mathrm{ev}_c(H)\not\subset Z$. (Note that such a point exists. Indeed, otherwise, the image of every $f$ in $H$ would be contained in $Z$, contradicting the definition of $H$.) In particular, for this choice of $c$, the morphism $\mathrm{ev}_c$ is generically finite onto its image, so that $\dim H\leq \dim X = 2$.  We claim that $\dim H = 1$. To prove our claim, we argue by contradiction. 

Suppose that $\dim H = 2$. Then, for every $c$ in $C$, since $\mathrm{ev}_c$ is generically finite onto its image, the morphism $\mathrm{ev}_c:H\to X$ is  dominant. Thus, since $X$ is of general type, by Kobayashi-Ochiai's finiteness theorem (Theorem \ref{thm:ko}), the set $\{\mathrm{ev}_c:H\to X \ | \ c\in C(k)\}$ is finite.  This implies that every $f$ in $H$ is constant, contradicting the fact that $H$ is contained in the moduli space of non-constant maps from $C$ to $X$. This shows that $\dim H \leq 1$. Since $H$ is positive-dimensional by assumption, we conclude that $H$ is an irreducible reduced (possibly singular and  non-proper) curve over $k$.

Again, let  $c$ in $C(k)$ be such that $\mathrm{ev}_c(H)\not\subset Z$.  Let $\overline{H}$ be the smooth projective model of $H$. Then, as $\overline{H}$ is smooth and one-dimensional, we have that $\mathrm{ev}_c:H\to X$ induces a morphism $\mathrm{ev}_c':\overline{H}\to X$. Now, since  $\mathrm{ev}_c':\overline{H}\to X$ does not map $\overline{H}$ into $Z$, it follows that $\overline{H}$ has genus at least two. Indeed, every curve of genus at most one mapping non-trivially to $X$ maps into $Z$ by our assumption that $X$ is geometrically hyperbolic modulo $Z$. This concludes the proof.
\end{proof}
 
 To prove Theorem \ref{thm:abstract}, we require several lemmas. 
 
 \begin{lemma}[Pseudo-groupless]\label{lem:groupless}
Let $X$ be a projective variety over a  field $K$ of characteristic zero, and let $Z\subset X $ be a proper closed subset. Suppose that, for every finitely generated field extension $L/K$ and every curve $C$ over $L$, the set of non-constant morphisms $f:C\to X_L$ with $f(C)\not\subset Z_L$ is finite. Then, for every abelian variety $A$ over $\overline{K}$, every non-constant rational map $A\dashrightarrow X_{\overline{K}}$ factors over $Z_{\overline{K}}$. 
 \end{lemma}
 \begin{proof}   We argue by contradiction.  Replacing $K$ by a finite field extension if necessary,   there is an abelian variety $A$ over $ K$, 
 a dense open  $U \subset A$ whose complement $A\setminus U$ is of codimension at least two in $A$, and    a non-constant morphism  $U\to X$ over $K$. Let $L/K$ be a finitely generated field extension  and let  $C\subset U_L$ be  a smooth curve  such that, for every $n>0$, we have $[n]C\subset U$. (Such a field extension $L$ and curve $C$ in $U_L$ exist since the codimension of $A\setminus U$ is at least two in $A$.) Then, the morphism $f_n: C\to X_L$ given by $c\mapsto f(nc)$ is a well-defined morphism. Note that the sequence of morphisms $f_n$ is an  infinite sequence of pairwise distinct non-constant morphisms from $C$ to $X_L$ with $f_n(C)\not\subset Z_L$. This contradicts our assumption.
 \end{proof}
 
 \begin{lemma} \label{lem:gt}   
 Let $X$ be a projective surface over a  field $K$ of characteristic zero, and let $Z\subset X $ be a proper closed subset. Suppose that, for every finitely generated field extension $L/K$ and every curve $C$ over $L$, the set of non-constant morphisms $f:C\to X_L$ with $f(C)\not\subset Z_L$ is finite. Then the following statements hold.
 \begin{enumerate}
 \item The projective variety $X$ is of general type.
 \item  If $V$ is a smooth projective variety over $K$ and $f:V\dashrightarrow X$ is a non-constant non-dominant rational map whose image does not lie in $Z$, then $f$ extends (uniquely) to a morphism $V\to X$.
 \end{enumerate} 
 \end{lemma}
 \begin{proof}
By Lemma \ref{lem:groupless}, for every abelian variety $A$ over $\overline{K}$,   every rational map $A\dashrightarrow X_{\overline{K}}$ factors over $Z_{\overline{K}}$. Therefore,  since $X$ is two-dimensional (by assumption),   it follows  that $X$ is of general type (use \cite[Corollary~3.17]{JXie} and \cite[Lemma~3.23]{JXie}).  This proves the first statement.
  
To prove the second statement, let $U\subset V$ be a dense open subset such that $f$ is regular on $U$. Let $C:= \overline{f(U)}$ be the (Zariski-)closure of the image of $U$. Note that $C$ is a  (geometrically integral) projective curve over $K$ (since $V\dashrightarrow X$ is non-dominant and non-constant). Since $f$ does not factor over $Z$, we have that $C\subset X$ does not lie in $Z$, so that the normalization of $C$ is a smooth projective curve of genus at least one. (Here we use that every non-constant morphism $\mathbb{P}^1_{\overline{K}}\to X_{\overline{K}}$ factors over $Z_{\overline{K}}$.) Now, since $C$ does not contain any rational curves, by \cite[\S 3]{JKa}, the rational map $V\dashrightarrow C$ extends to a morphism $V\to C$. This concludes the proof.
 \end{proof}

\begin{lemma}\label{lemma:cutting}
Let $X$ be a projective surface over a field $K$ of characteristic zero, and let $Z\subset X $ be a proper closed subset. 

Suppose that, for every finitely generated field extension $L/K$ and every smooth projective curve $C$ over $L$, the set of non-constant morphisms $f:C\to X_L$ with $f(C)\not\subset Z_L$ is finite.

 Then, for every finitely generated field extension $M/K$ and every variety $V$ over $M$, the set of non-constant rational maps $g:V\dashrightarrow X_M$ with $g(V)\not\subset Z_M$ is finite.
\end{lemma}
\begin{proof}  We may assume that $M=K$, and that $V$ is a smooth projective variety over $K$. To prove the lemma, we argue by induction on $\dim V$.

 If $\dim V=1$, let $\overline{V}$ be a smooth projective model for $V$. Then, every rational map $V\dashrightarrow X$ extends uniquely to a morphism $\overline{V}\to X$. Therefore,    the finiteness statement holds by assumption.
  
  Assume that $\dim V> 1$.  Suppose that $f_i:V\dashrightarrow X$ with $i\in \mathbb{N}$ are pairwise distinct non-constant rational maps over $K$.
Since $X$ is of general type (by Lemma \ref{lem:gt}.(1)), we can apply the theorem of Kobayashi-Ochiai (Theorem \ref{thm:ko}) to conclude that all but finitely many of the morphisms $f_i$ are non-dominant. Thus, discarding finitely many of the $f_i$ if necessary,  each $f_i:V\dashrightarrow X$ is a non-constant non-dominant rational map whose image does not lie in $Z$. Therefore, by Lemma \ref{lem:gt}.(2), each $f_i:V\dashrightarrow X$ extends to a morphism $f_i:V\to X$.

Let $K\subset k$ be an uncountable algebraically closed field extension.  
  For each $i$ and $j$, let $V^{i,j} =\{P \in X(k) \ | \ f_i(P) \neq f_j(P)\}$.  
 Also,  for each $i$, let $\Delta_i := f_i^{-1}(Z)$, and let $\Delta_i^0$ be its complement.
 Since $k$ is uncountable, there exists a point $P\in X(k)$ in the countable intersection of non-empty open subsets
\[
(\cap_{i\neq j} V^{i,j}) \cap (\cap_i \Delta_i^0).
\]
Let $\mathcal{H}\subset V_k$ be a smooth irreducible ample divisor containing $P$. There is a finitely generated field extension $L/K$ such that $P$ and $\mathcal{H}\subset V_k$ can be defined over $L$, i.e., the point $P$ lies in $V(L)$ and there is a smooth irreducible ample divisor $\ {H}\subset V_L$ containing $P$.  Define $g_i:=f_i|_H:H\to X$ to be the restriction of $f_i$ to $H$.  Note that $g_i(P)=f_i(P) \neq f_j(P)= g_j(P)$, so that the morphisms $g_i$ are pairwise distinct. Also, since $P\in \Delta_i^0$, we have that $g_i(H)\not\subset Z_L$.

Note that, for each $i$, the morphism $g_i:H\dashrightarrow X$ is non-constant.  To prove this,  suppose that   $g_i(H) = \{p\}$. Choose  $q\neq p$ in $g_i(V)\setminus Z$ and note that the (positive-dimensional) fibre $V_q$ does not intersect $H$. This contradicts the ampleness of $H$.
Thus, the sequence $(g_i:H\to X_L)$ is an infinite sequence of pairwise distinct non-constant morphisms with $g_i(H)\not\subset Z_L$; this contradicts the induction hypothesis and concludes the proof.
\end{proof}

 \begin{proof}[Proof of Theorem \ref{thm:abstract}]
By assumption,  there is an embedding  $K\to \mathbb{C}$ such that, for  every smooth projective curve $C$  over $\CC$,  the scheme $\underline{\Hom}_{\CC}(C,X_{\CC})\setminus \underline{\Hom}_{\CC}(C, Z_{\CC})$ is of finite type over $\CC$. 

Note that $Z_{\CC}$ contains all rational curves of $X_{\CC}$. Indeed, if there was a rational curve $\mathbb{P}^1_{\CC}\to X_{\CC}$ not factoring over $Z_{\CC}$, then the  scheme $\underline{\Hom}_{\CC}(\mathbb{P}^1_{\CC},X_{\CC})\setminus \underline{\Hom}_{\CC}(\mathbb{P}^1_{\CC}, Z_{\CC})$ would have infinitely many connected components, contradicting the fact that the latter scheme is of finite type. (The existence of infinitely many components follows from the existence of  self-maps of degree $>1$ on $\mathbb{P}^1_{\CC}$.) It follows from a standard specialization argument (Lefschetz  principle) that, for every algebraically closed field $k$  containing $k$, the subset $Z_k\subset X_k$ contains all rational curves of $X_k$.

A  standard specialization argument also shows that  our assumption persists over all algebraically closed fields of characteristic zero (Lefschetz principle). More precisely,  for every algebraically closed field $k$ containing $K$, for every smooth projective curve $C$ over $k$, the scheme $\underline{\Hom}_k(C,Xk)\setminus\underline{\Hom}_k(C,Z_k)$ is of finite type over $k$; we refer the reader to \cite[\S 9]{vBJK} for details.

By Mori's bend-and-break \cite[Proposition~3.1]{Debarrebook1}, for any algebraically closed field $k$ containing $K$, for any smooth projective curve $C$ over $k$, every $c$ in $C(k)$, and every $x$ in $X(k)\setminus Z(k)$, the set of pointed morphisms $(C,c)\to (X_{k}, x)$ is finite. Indeed, if this set were infinite, then we would have an infinite sequence of pointed maps $(C,c)\to (X_{k},x)$ of bounded degree (as $\underline{\Hom}_k(C,X_k)\setminus \underline{\Hom}_k(C,Z_k)$ is of finite type over $k$), so that there is a rational curve passing through $x$. Since $x\not\in Z_{k}$, this contradicts the fact that $Z_{k}$ contains all rational curves of $X_{k}$.

By Theorem \ref{thm:abstract2}, for every algebraically closed field $k$ containing $K$ and every smooth projective curve $C$ over $k$, we have that every positive-dimensional irreducible component of  
   \[\underline{\Hom}_{k}^{nc}(C,X_k\})\setminus \underline{\Hom}_{k}(C, Z_{k})\] is birational to  a smooth projective curve of genus at least two over $k$.

   Let $L/K$ be a finitely generated field extension,  let $k$ be an algebraically closed field containing $L$, and let $C$ be a smooth projective curve over $L$. To conclude the proof, by Lemma \ref{lemma:cutting}, it suffices to show that  the set of non-constant morphisms $f:C\to X_L$ with $f(C)\not\subset Z_L$ is finite.
   
   To do so, note that, since $\underline{\Hom}_k^{nc}(C_k,X_k)\setminus \underline{\Hom}_k(C_k,Z_k)$ is of finite type, the scheme $\mathcal{M}:= \underline{\Hom}^{nc}_L(C,X_L)\setminus \underline{\Hom}_k(C,Z_L)$   is of finite type over $L$.     In particular, since (the reduced closed subscheme of) every positive-dimensional irreducible component of  $\mathcal{M}_k$ is birational to a smooth projective curve of genus at least two, it follows from  Faltings's theorem \cite{FaltingsComplements} that the set $\mathcal{M}(L) $ of non-constant morphisms $f:C\to X_L$ with $f(C)\not\subset Z_L$ is finite, as required. 
 \end{proof}

The three applications stated in the introduction are   now obtained by combining  Theorem \ref{thm:abstract} with well-known theorems of Bogomolov and Kobayashi.

\begin{proof}[Proof of Theorem \ref{thm:brody}]
Let $K\to \CC$ be an embedding such that $X_{\CC}$ is Brody hyperbolic. Then, $X_{\CC}$ is $1$-bounded over $\CC$; see \cite[Theorem~5.3.9]{KobayashiBook}. Thus,  the statement follows from Theorem \ref{thm:abstract} (with "$Z=\emptyset$").
\end{proof}

\begin{proof}[Proof of Theorem \ref{thm:omega1}]  Let $K\to \mathbb{C}$ be an embedding.
If $\Omega^1_X$ is ample, then $\Omega^1_{X_\CC}$ is ample. In particular,  the variety $X_{\CC}$ is  Brody hyperbolic \cite[Theorem~3.6.21]{KobayashiBook}, so that  the statement follows from Theorem  \ref{thm:brody}.
\end{proof}

\begin{proof}[Proof of Theorem \ref{thm:bogomolov}] As in the statement of the theorem, let $X$ be a smooth projective surface over $K$ with $c_1^2 > c_2$.  Let $K\to \mathbb{C}$ be an embedding. Then, 
by Bogomolov's theorem \cite{Bogomolov},  there is a proper closed subset $Z\subset X_{\CC}$ such that, for every smooth projective curve $C$  over $\CC$, the scheme $\underline{\Hom}_{\CC}(C,X_\CC)\setminus \underline{\Hom}_{\CC}(C,Z)$ is of finite type. (Equivalently, the degree of a morphism $f:C\to X$ with $f(C)\not\subset  Z$ is bounded by a constant depending only on $X,Z, C$.) Therefore, the statement follows from Theorem \ref{thm:abstract}.
\end{proof}

% \begin{remark}
 %Let $X\subset \Gamma\backslash D$ be a projective two-dimensional subvariety of a locally symmetric variety. Then, as $X$ is Brody hyperbolic, one can apply Theorem \ref{thm:brody}.   This applies, for example, to two-dimensional subvarieties of exceptional Shimura varieties.
% \end{remark}
 
 \begin{remark}
 We can   apply Theorem \ref{thm:abstract}  to any pseudo-algebraically hyperbolic projective surface. For example, if $X$ is a  projective normal surface of general type which has maximal Albanese dimension, then $X$ is pseudo-algebraically hyperbolic by \cite{Yamanoi1} (see \cite[Theorem~3.9]{JR} for details). One can also apply Theorem \ref{thm:abstract} to two-dimensional complete subvarieties of certain locally symmetric varieties.
 \end{remark}

 \begin{remark}[Persistence Conjecture]\label{remark:pers}
 Let $X$ be a   projective variety over a number field $K$ such that, for every finite extension $L/K$, the set of $L$-rational points $X(L)$ is finite. Lang asked \cite[p.~202]{Lang2} whether $X(M)$ is finite for every finitely generated extension $M/K$; this question probably has a positive answer and is formulated as the ``Persistence Conjecture'' in \cite[Conjecture~1.15]{vBJK}. Our finiteness results in this note are  also motivated by  the Persistence Conjecture. To explain this, consider a projective surface $X$ over a number field $K$ such that $X_{\CC}$ is Brody hyperbolic with respect to some embedding $K\to \mathbb{C}$.   Recall that such varieties $X$ over $K$ are conjectured to have only finitely many $L$-points for each finite extension $L/K$. In \cite[Theorem~1.7]{JAut} it is shown that, if $X(L)$ is finite for every number field $L/K$, then $X(M)$ is finite for every finitely generated field extension $M/K$ (so that Lang's conjecture has a positive answer in this case).  Although the finiteness of $X(L)$ is currently not known for number fields $L$,  our result (Theorem \ref{thm:brody}) verifies that, for $M/K$ a finitely generated extension of positive transcendence degree,  the subset   of "non-constant" elements of $X(M)$ is finite, as predicted by the Persistence Conjecture. 
\end{remark}

\section{Ramified covers} Let $k$ be  an algebraically closed field of characteristic zero.
Let $X$ be a projective variety of general type over $k$. Lang conjectured that $X$ is pseudo-Mordellic over $k$, i.e., there exists a proper closed subset $\Delta\subsetneq X$ such that, for every finitely generated subfield $K\subset k$ and every model $\mathcal{X}$ for $X$ over $K$, the set $\mathcal{X}(K)\setminus \Delta$ is finite. 

Assuming $X$ embeds into an abelian variety, this conjecture (commonly referred to as the Mordell-Lang conjecture) was proven by Faltings \cite{FaltingsLang1}. Faltings proved a more precise result. In fact, 
 if $X$ is a closed subvariety  of an abelian variety $A$ over $k$, we define  the special locus $\mathrm{Sp}(X)$ of $X$ to be  the union of the translates of positive-dimensional abelian subvarieties of $A$ contained in $X$.  By work of Kawamata and Ueno \cite{Kawamata, Ueno}, the subset  $\mathrm{Sp}(X)$ is a closed subset of $X$ and $X$ is  of general type if and only if $\mathrm{Sp}(X)\neq X$. Faltings proved that $X$ is Mordellic modulo $\mathrm{Sp}(X)$, i.e.,   for every finitely generated field $K$ contained in $k$ and every model $\mathcal{X}$ for $X$ over $K$, the set $\mathcal{X}(K)\setminus \mathrm{Sp}(X)$ is finite.
 
 Faltings's more precise result   implies that every proper closed subvariety of a  simple abelian variety is Mordellic (because the special locus of any such subvariety is obviously empty). We will use this consequence of Faltings's work in  the following form in our proof of Theorem \ref{thm:ram}.
 
 \begin{proposition}[Faltings]\label{prop:faltings}
 Let $A$ be a simple abelian variety over an algebraically closed field $k$ of characteristic zero, and let $H\to A$ be a finite   morphism of schemes. If $\dim H <\dim A$, then $H$ is Mordellic over $k$, i.e., for every finitely generated subfield $K\subset k$ and every model $\mathcal{X}$ for $X$ over $K$, the set $\mathcal{X}(K)\ $ is finite. 
 \end{proposition}
 \begin{proof} Let $H'$ be the image of $H\to A$. Note that $\dim H' = \dim H < \dim A$. (In particular, the morphism $H\to A$ is not surjective.) Thus, 
 by Faltings's theorem, since $\dim H'<\dim A$ and $A$ is simple, the projective variety $H'$ is Mordellic over $k$. Since $H$ is finite over $H'$ and $H'$ is Mordellic, we conclude that $H$ itself is Mordellic.
 \end{proof}

\begin{proof}[Proof of Theorem \ref{thm:ram}] Let $k=\overline{K}$ be an algebraic closure of $K$. Since $X_k$ is an integral normal projective variety and $X_k\to A_k$ is a non-\'etale finite surjective morphism, by \cite{Yamanoi1}, the projective variety $X_k$ is  algebraically hyperbolic over $k$  (as defined by  \cite{Demailly}, see also \cite{JKa}). Now, let $L/K $ and $V$ be as in the statement of the theorem, and fix an embedding $L\to k$. To prove the desired finiteness statement, we may and do assume that $V$ is a smooth projective variety over $L$. Since $X_{k}$   has no rational curves,  every rational map $V\dashrightarrow X_L$ extends (uniquely) to a morphism $V\to X_L$; see \cite[\S 3]{JKa}. By \cite[Theorem~1.9]{JKa}, the scheme $H:=\underline{\Hom}_k^{nc}(V_k,X_k)$ parametrizing non-constant morphisms from $V_k$ to $X_k$ is a projective algebraically hyperbolic scheme of dimension at most $\dim X - 1 = \dim A -1$. Also, given $v$ in $V$, the evaluation morphism $\mathrm{ev}_v: H\to X$ has finite fibres \cite[Theorem~1.8]{JKa}, so that (by projectivity), the evaluation morphism $\mathrm{ev}_v$ is a finite morphism (see, for example, \cite[Tag~01WG]{stacks-project}). In particular, the composed morphism $H\to X\to A$ is finite, so that $H$ is a Mordellic projective scheme over $k$ by Faltings's theorem (Proposition \ref{prop:faltings}). 
This implies the desired finiteness statement. 
\end{proof}

  \bibliography{refsperiod}{}

\newcommand{\etalchar}[1]{$^{#1}$}
\def\cprime{$'$}
\begin{thebibliography}{{Sta}15}

\bibitem[BJK]{vBJK}
R.~van Bommel, A.~Javanpeykar, and L.~Kamenova.
\newblock Boundedness in families with applications to arithmetic
  hyperbolicity.
\newblock {\em arXiv:1907.11225}.

\bibitem[Bog77]{Bogomolov}
F.~A. Bogomolov.
\newblock Families of curves on a surface of general type.
\newblock {\em Dokl. Akad. Nauk SSSR}, 236(5):1041--1044, 1977.

\bibitem[CDJ{\etalchar{+}}]{CDJLZ}
P.~Corvaja, J.L. Demeio, A.~Javanpeykar, D.~Lombardo, and U.~Zannier.
\newblock On the distribution of rational points on ramified covers of abelian
  varieties.
\newblock {\em arXiv:2011.12840}.

\bibitem[Deb01]{Debarrebook1}
O.~Debarre.
\newblock {\em Higher-dimensional algebraic geometry}.
\newblock Universitext. Springer-Verlag, New York, 2001.

\bibitem[Dem97]{Demailly}
J.-P. Demailly.
\newblock Algebraic criteria for {K}obayashi hyperbolic projective varieties
  and jet differentials.
\newblock In {\em Algebraic geometry---{S}anta {C}ruz 1995}, volume~62 of {\em
  Proc. Sympos. Pure Math.}, pages 285--360. Amer. Math. Soc., Providence, RI,
  1997.

\bibitem[EJR]{EJR}
A~Etesse, A.~Javanpeykar, and E.~Rousseau.
\newblock Algebraic intermediate hyperbolicities.
\newblock {\em arXiv:2012.07803}.

\bibitem[Fal84]{FaltingsComplements}
G.~Faltings.
\newblock Complements to {M}ordell.
\newblock In {\em Rational points ({B}onn, 1983/1984)}, Aspects Math., E6,
  pages 203--227. Vieweg, Braunschweig, 1984.

\bibitem[Fal91]{FaltingsLang1}
G.~Faltings.
\newblock Diophantine approximation on abelian varieties.
\newblock {\em Ann. of Math. (2)}, 133(3):549--576, 1991.

\bibitem[Gro62]{FGA}
A.~Grothendieck.
\newblock {\em Fondements de la g\'{e}om\'{e}trie alg\'{e}brique. [{E}xtraits
  du {S}\'{e}minaire {B}ourbaki, 1957--1962.]}.
\newblock Secr\'{e}tariat math\'{e}matique, Paris, 1962.

\bibitem[Jav20]{JBook}
A.~Javanpeykar.
\newblock The {L}ang-{V}ojta conjectures on projective pseudo-hyperbolic
  varieties.
\newblock In {\em Arithmetic geometry of logarithmic pairs and hyperbolicity of
  moduli spaces}, CRM Short Courses, pages 135--196. Springer, Cham, [2020]
  \copyright 2020.

\bibitem[Jav21]{JAut}
A.~Javanpeykar.
\newblock Arithmetic hyperbolicity: automorphisms and persistence.
\newblock {\em Math. Ann.}, 381(1-2):439--457, 2021.

\bibitem[JK]{JKa}
A.~Javanpeykar and L.~Kamenova.
\newblock Demailly's notion of algebraic hyperbolicity: geometricity,
  boundedness, moduli of maps.
\newblock {\em Mathematische Zeitschrift volume 296, 1645-1672, (2020)}.

\bibitem[JR]{JR}
A.~Javanpeykar and E.~Rousseau.
\newblock Albanese maps and fundamental groups of varieties with many rational
  points over function fields.
\newblock {\em IMRN, to appear. arXiv:2010.02913}.

\bibitem[JX]{JXie}
A.~Javanpeykar and J.~Xie.
\newblock Finiteness properties of pseudo-hyperbolic varieties.
\newblock {\em IMRN, to appear. arXiv:1909.12187}.

\bibitem[Kaw80]{Kawamata}
Y.~Kawamata.
\newblock On {B}loch's conjecture.
\newblock {\em Invent. Math.}, 57(1):97--100, 1980.

\bibitem[Kob98]{KobayashiBook}
S.~Kobayashi.
\newblock {\em Hyperbolic complex spaces}, volume 318 of {\em Grundlehren der
  Mathematischen Wissenschaften [Fundamental Principles of Mathematical
  Sciences]}.
\newblock Springer-Verlag, Berlin, 1998.

\bibitem[Lan86]{Lang2}
S.~Lang.
\newblock Hyperbolic and {D}iophantine analysis.
\newblock {\em Bull. Amer. Math. Soc. (N.S.)}, 14(2):159--205, 1986.

\bibitem[Nit05]{Nitsure}
N.~Nitsure.
\newblock Construction of {H}ilbert and {Q}uot schemes.
\newblock In {\em Fundamental algebraic geometry}, volume 123 of {\em Math.
  Surveys Monogr.}, pages 105--137. Amer. Math. Soc., Providence, RI, 2005.

\bibitem[{Sta}15]{stacks-project}
The {Stacks Project Authors}.
\newblock \emph{{S}tacks {P}roject}.
\newblock http://stacks.math.columbia.edu, 2015.

\bibitem[Uen75]{Ueno}
K.~Ueno.
\newblock {\em Classification theory of algebraic varieties and compact complex
  spaces}.
\newblock Lecture Notes in Mathematics, Vol. 439. Springer-Verlag, Berlin-New
  York, 1975.
\newblock Notes written in collaboration with P. Cherenack.

\bibitem[Yam15]{Yamanoi1}
K.~Yamanoi.
\newblock Holomorphic curves in algebraic varieties of maximal {A}lbanese
  dimension.
\newblock {\em Internat. J. Math.}, 26(6):1541006, 45, 2015.

\end{thebibliography}
\bibliographystyle{alpha}

\end{document}